\documentclass[a4paper]{amsart}

\usepackage{amssymb,latexsym}
\usepackage[utf8x]{inputenc}
\usepackage{amsmath,amsthm}
\usepackage{amsfonts,mathrsfs}
\usepackage{enumerate,units}
\usepackage[all]{xy}
\usepackage{graphicx}

\theoremstyle{plain}
\newtheorem{theorem}{Theorem}[subsection]

\newtheorem{lemma}[theorem]{Lemma}
\newtheorem{proposition}[theorem]{Proposition}

\newtheorem{fact}[theorem]{Fact}
\newtheorem*{claim}{Claim}
\newtheorem*{theorem*}{Theorem}

\theoremstyle{definition}
\newtheorem{definition}[theorem]{Definition}

\theoremstyle{remark}
\newtheorem*{remark}{Remark}
\newtheorem*{example}{Example}
\newtheorem*{notation}{Notation}

\numberwithin{equation}{section}
\newcommand{\forkindep}[1][]{%
  \mathrel{
    \mathop{
      \vcenter{
        \hbox{\oalign{\noalign{\kern-.3ex}\hfil$\vert$\hfil\cr
              \noalign{\kern-.7ex}
              $\smile$\cr\noalign{\kern-.3ex}}}
      }
    }\displaylimits_{#1}
  }
}

\newenvironment{claimproof}[1][\proofname]
  {%
    \proof[#1]%
  }
  {%
    \endproof%
  }

\newcounter{step}                   
    {\hfill $\clubsuit$             
     \vspace{7pt}\par}
  
\begin{document}
\title{$\infty$-constructible Subsemigroups of $M_2(\mathbb{C})$}
\author{Yatir Halevi}
\thanks{The research leading to these results has received funding from the European Research Council under the European Union's Seventh Framework Programme (FP7/2007-2013)/ERC Grant Agreement No. 291111.}
\address{Einstein Institute of Mathematics, The Hebrew University of Jerusalem, Givat Ram 9190401, Jerusalem, Israel}
\email{yatir.halevi@mail.huji.ac.il}

\maketitle

\begin{abstract}
A description of all subsemigroups of $M_2(\mathbb{C})$ which are given by a countable intersection of constructible sets is given. Furthermore, it is shown that they are intersections of constructible semigroups.
\end{abstract}

\section{Introduction}
A constructible set $X\subseteq \mathbb{C}^n$ is a finite union of locally closed sets in the Zariski topology. It is well known that every constructible subgroup of an algebraic group over $\mathbb{C}$ is closed (i.e. an algebraic group) \cite[Lemma 2.2.4]{springer}. Actually, every $\infty$-constructible (i.e. a countable intersection of constructible sets) subgroup of an algebraic group is also an algebraic group (see Section \ref{ss:mt-of-C}). Algebraic subgroups of $M_2(\mathbb{C})$ were characterized by Nguyen, van der Put and Top in \cite{Nguyen}. 

The aim of this paper is to generalize these results to $\infty$-constructible subsemigroups of $M_2(\mathbb{C})$. The main theorem is
\begin{theorem*}
Let $M=H\cup S\subseteq M_2(\mathbb{C})$ be an $\infty$-constructible subsemigroup, where $H$ is the subgroup of invertible matrices and $S$ the subsemigroup of singular matrices. Then $M$ is an intersection of countably many constructible semigroups. Moreover, if $\pi(H)$ is infinite, where $\pi:GL_2(\mathbb{C})\to PGL_2(\mathbb{C})$ is the natural surjection, then $M$ is constructible.
\end{theorem*}

Note that, as opposed to the case of groups, not every constructible subsemigroup of $M_2(\mathbb{C})$ is closed. For instance consider all the invertible matrices of $M_2(\mathbb{C})$.

Preceding the proof this theorem, we give a characterization of the subsemigroups of $M_2(\mathbb{C})$. We continue to give an outline of the paper. 

In Section \ref{s:PM2}, we show that any subsemigroup of $M_2(\mathbb{C})/\mathbb{C}^\times$ is a union of an algberaic subgroup of $PGL_2(\mathbb{C})$ and a semigroup which is essentially a combinatorical object. Namely, since every nonzero singular element of $M_2(\mathbb{C})/\mathbb{C}^\times$ is determined by its kernel and image, we may identify $M_2(\mathbb{C})/\mathbb{C}^\times$ with $(\mathbb{C}P^1)^2$, where $\mathbb{C}P^1$ is the projective line. Under this identification the semigroup operation becomes
\begin{equation*}
(v,u)\cdot (v',u') =
\begin{cases}
(v,u') & \text{if }u\neq v' \text{ and }\\
0 & \text{if } u=v'.
\end{cases}
\end{equation*}
We show that essentially every subsemigroup of $M_2(\mathbb{C})/\mathbb{C}^\times$ corresponds, under this identification, to a set of the form
\[\{(v,u):v\in F,u\in G\}\cup\{0\},\] for some $F,G\subseteq \mathbb{C}P^1$.

In Section \ref{s:M_2}, we study subsemigroups of $S\subseteq M_2(\mathbb{C})$ by studying their images in $M_2(\mathbb{C})/\mathbb{C}^\times$. We associate with every singular element $a\in S$ of the semigroup a certain multiplicity, \[Z_a=\{z\in\mathbb{C}^\times: za\in S\}\] and study how it varies in the semigroup. The results of this section, concluding with Propositions \ref{P:M_2-type A} and \ref{P:M_2-type B}, prove the Main Theorem.

As a last remark, in \cite[Proposition 2.7, Remark 2.8]{on_enveloping}, Milliet gives a purely model theoretic proof which implies that every $\infty$-constructible subsemigroup of $M_n(\mathbb{C})$ (for any $n$) is an intersection of constructible semigroups. However, his result does not say anything about the algebraic structure of these semigroups. 

All of what we do may be done over any uncountable algebraically closed field of characteristic zero.

\section{Preliminaries and Notation}\label{S:pre_not}
\subsection{Notation}
We start with some preliminaries from semigroup theory. A semigroup is a set $S$ together with an associative binary operation. An idempotent is an element $e\in S$ satisfying $e^2=e$ and a nilpotent $n\in S$ is an element satisfying $n^2=0$ (which only makes sense if $S$ has a zero element: an element $0\in S$ such that $0\cdot a=a\cdot 0=0$ for all $a\in S$). Denote by $E(S)$ the set of idempotents of $S$.  We will be using the unorthodox (but will make our writing easier) convention that the zero element is not an idempotent.

Let $M_2(\mathbb{C})$ be the monoid of $2\times 2$ matrices over $\mathbb{C}$, $M_2^0(\mathbb{C})$ the subsemigroup of singular matrices, $GL_2(\mathbb{C})$ the subgroup of invertible matrices and $PM_2(\mathbb{C})$ the monoid $M_2(\mathbb{C})/\mathbb{C}^\times$ similarly $PM_2^0(\mathbb{C})$ and $PGL_2(\mathbb{C})$.
Let $\pi : M_2(\mathbb{C})\to PM_2(\mathbb{C})$ be the natural surjection.

\subsection{Model Theory of $\mathbb{C}$}\label{ss:mt-of-C}
Due to quantifier elimination, constructible sets arise naturally in the model theory of algebraically closed fields. We recall some results, most of the following may be found in \cite{Marker02}.

Let $\mathcal{L}$ be a first order language and $T$ a complete consistent theory over $\mathcal{L}$. We will usually write $x$ (one variable) instead of $\overline{x}$ and the same for parameters ($a$ instead of $\overline{a}$).

\begin{definition}
Let $M$ be an $\mathcal{L}$-structure. 
A subset $X\subseteq M^k$ is \emph{definable} over $A$, for $A\subseteq M$ if there exists an $\mathcal{L}_A$-formula $\psi(x)$ such that $\psi(M)=X$ and $\infty$-definable if it is an intersection of definable sets.
\end{definition}

Let ACF$_0$ be the theory of algebraically closed fields of characteristic zero in the language $\mathcal{L}=\{+,-,\cdot,0,1\}$ of rings. $\mathbb{C}$ is a model of this theory. It is a complete consistent theory and enjoys some very nice model theoretic properties:

\begin{fact}\label{F:QE}\cite[Theorem 3.2.2]{Marker02}
ACF$_0$ has quantifier elimination, i.e. for every formula $\phi$ there exists a quantifier free formula $\psi$ such that they define the same definable set. Thus every definable set in $\mathbb{C}$ corresponds to a constructible set, in the algebraic geometry sense, i.e. a finite union of locally closed sets.
\end{fact}

\begin{definition}
A subset of $\mathbb{C}^n$ will be called \emph{$\infty$-constructible} if it is a countable intersection of constructible sets.
\end{definition}

\begin{fact}\cite[Exercise 4.5.17]{Marker02}\label{F:saturated}
$\mathbb{C}$ is $\aleph_1$-saturated, i.e. for every countable family of constructible subsets $\{C_i\}_{i< \omega}$ of $\mathbb{C}$, if for every finite $I\subseteq \omega$, $\bigcap_{i\in I} C_i\neq \emptyset$ then $\bigcap_{i< \omega} C_i\neq \emptyset$. 
\end{fact}

As a result, we have the following generalization of Chevalley's Theorem.
\begin{lemma}\label{L:gen-of-chev}
Let $V$ and $W$ be a varieties over $\mathbb{C}$ and $f:V\to W$ a morphism. If $C\subseteq V$ is ($\infty$-)constructible then so is $f(V)$. Conversely, if $C\subseteq W$ is ($\infty$-)constructible then so is $f^{-1}(W)$.
\end{lemma}
\begin{remark}
Every variety over $\mathbb{C}$, and morphisms between varieties over $\mathbb{C}$, may be interpreted as constructible sets in $\mathbb{C}$ (see \cite[Section 7.4]{Marker02}).
\end{remark}
\begin{proof}
If $C\subseteq V$ is constructible then $f(C)$ is constructible by the regular Chevalley's Theorem (this also follows by quantifier elimination). Assume $C=\bigcap_{i<\omega} C_i$ is $\infty$-constructible. We may assume that $\bigcap_{i<\omega} C_i$ is closed under finite intersections, and let $a\in \bigcap_{i<\omega} f(C_i)$. We need to show that the following is non-empty
\[\{x:f(x)=a\}\cap\bigcap_{i<\omega}C_i,\] and indeed this is true by Fact \ref{F:saturated}.

The converse is true, since the pre-image of every constructible set is constructible.
\end{proof}

\begin{lemma}\label{L:cons-in-curve}
Let $V$ be an affine integral curve over $\mathbb{C}$. Every $\infty$-constructible subset $C\subseteq V$ is either finite or co-countable.
\end{lemma}
\begin{proof}
By Noether's normalization there exists a map $f:V\to \mathbb{A}_\mathbb{C}^1$ with finite fibers. Hence it is enough to show it for $\mathbb{A}_{\mathbb{C}}^1$. Since every constructible subset of $\mathbb{A}^1_\mathbb{C}$ is either finite or cofinite, we obtain our result.
\end{proof}

We end with some results concerning groups and semigroups:

\begin{fact}\label{F:subconstgrp-of-alg}\cite[Theorem 7.5.3, Lemma 7.4.9] {Marker02}
Every $\infty$-constructible subgroup of an algebraic groups is a closed algebraic subgroup.
\end{fact}

Using the above and \cite[Lemma 7.5.2]{Marker02}, we get the following:
\begin{fact}\label{F:subsconstsemi-of-alg}
Every $\infty$-constructible subsemigroup of an algebraic group is a closed algebraic subgroup.
\end{fact}

\section{Submonoids of $PM_2(\mathbb{C})$}\label{s:PM2}

In this section we describe all the submonoids of $PM_2(\mathbb{C}=M_2(\mathbb{C
})$.
Every submonoid $M\subseteq PM_2(\mathbb{C})$ may be decomposed as $M=H\cup S$ where $H\subseteq PGL_2(\mathbb{C})$ and $S\subseteq PM_2^0(\mathbb{C})$.

\subsection{Subsemigroups of $PM^0_2(\mathbb{C})$}
All the non-zero singular matrices have rank one and hence we can identify $PM_2^0(\mathbb{C})$ with $\left( \mathbb{C}P^1\right)^2$, where $\mathbb{C}P^1$ is the complex projective line, the first coordinate corresponds to the image and the second to the kernel.
Corresponding to the semigroup operation of $M_2(\mathbb{C})$, we have the following multiplication law for the non-zero elements of $PM_2^0(\mathbb{C})$:
\begin{equation*}
(v,u)\cdot (v',u') =
\begin{cases}
(v,u') & \text{if }u\neq v' \text{ and }\\
0 & \text{if } u=v'.
\end{cases}
\end{equation*}

\begin{remark}
Note that each element is either an idempotent or nilpotent.
\end{remark}

\begin{definition} 
For each $F,G\subseteq \mathbb{C}P^1$, define the following semigroups:
\[B_{F,G}=\{(v,u):v\in F,u\in G\}\cup \{0\}.\]
\end{definition}
If $F,G$ are finite or cofinite then $B_{F,G}$ is constructible in $PM_2^0(\mathbb{C})$.

\begin{lemma}\label{L:F,G as const}
If $B_{F,G}$ is $\infty$-constructible then so are $F$ and $G$, as subsets of $\mathbb{C}P^1$. Moreover, $F$ and $G$ are finite or co-countable.
\end{lemma}
\begin{proof}
Viewing $B_{F,G}$ as an $\infty$-constructible subset of $(\mathbb{C}P^1)^2$, $F$ (resp. $G$) is equal to the projection on the first (resp. second) coordinate. By Lemma \ref{L:gen-of-chev}, $F$ (resp. $G$) is an $\infty$-constructible subset of $\mathbb{C}P^1$ and the moreover part follows using Lemma \ref{L:cons-in-curve}.
\end{proof}

Essentially, every subsemigroup of $PM_2^0(\mathbb{C})$ is of this form:

\begin{proposition}\label{P:first}
Let $S$ be a subsemigroup of $PM_2^0(\mathbb{C})$ with zero.  
$S$ is one of the following:
\begin{itemize}
\item{Type A:} $B_{F,G}$;
\item{Type B:} $B_{F,\{v\}}\cup B_{\{v\},G}$ (for some $v\in F\cap G$).
\end{itemize}

If $S$ does not have a zero, necessarily $S=B_{F,G}\setminus \{0\}$ for $F\cap G=\emptyset$.

Moreover, $S$ is an intersection of (possibly uncountably many) constructilbe subsemigroups of $PM_2^0(\mathbb{C})$. If $S$ is $\infty$-constructible then it is an intersection of countably many constructible semigroups.
\end{proposition}

\begin{proof}
Let
\[S_1=\{v\in \mathbb{C}P^1:\exists u\in  \mathbb{C}P^1, (v,u)\in S\},\]
\[S_2=\{u\in  \mathbb{C}P^1:\exists v\in  \mathbb{C}P^1, (v,u)\in S\},\]
\[L=\{v\in S_1:\exists (u_1\neq u_2) \text{ s.t. } (v,u_1)\in S \text{ and }  (v,u_2)\in S\} \text{ and }\]
\[R=\{u\in S_2:\exists (v_1\neq v_2) \text{ s.t }(v_1,u)\in S \text{ and }(v_2,u)\in S\}.\]

We may assume that $S_1,S_2\neq \emptyset$.
If $|S_1|=1$ or $|S_2|=1$ then $S=B_{S_1,S_2}$  (Type A). Assume that $|S_1|>1$ and $|S_2|>1$.

\begin{claim}
\begin{enumerate}
\item If $L= \emptyset$ or $R= \emptyset$ then $S$ is of Type B.
\item If $v\in L$ then $(v,u)\in S$ for all $u\in S_2$ and the same for $R$ and $S_1$. So \[B_{L,S_2}\cup B_{S_1,R}\subseteq S,\]
with equality if $L\neq \emptyset$ and $R\neq \emptyset$.
\end{enumerate}
\end{claim}
\begin{claimproof}
\begin{enumerate}
\item Assume that $L=\emptyset$ and let $(v,u),(x,y)\in S$. If $u\neq x$ then $(v,y)\in S$ and so $y=v$ contradicting the fact that $|S_2|>1$. Similarly for $y\neq v$. Hence $S=\{(v,u),(u,v)\}\cup \{0\}$, which is of Type B.
\item Let $(v,u_1),(v,u_2)\in S$ such that $u_1\neq u_2$ and let $u\in S_2$. There exists $v'$ such that $(v',u)\in S$. We may assume that $v'\neq u_1$ and hence \[(v,u_1)(v',u)=(v,u)\in S.\]
Assume that $x\in L$ and $y\in R$ and let $(v,u)\in S$. If $v\notin L$ then similarly to what was done in $(1)$, $u\in R$ and hence $(v,u)\in B_{S_1,R}$. If $u\neq R$ then $v\in L$.
\end{enumerate}
\end{claimproof}

Assume that $S=B_{L,S_2}\cup B_{S_1,R}$ and that $L,R\neq \emptyset$.
If $L=R=\{v\}$ then 
\[S=B_{\{v\},S_2}\cup B_{S_1,\{v\}},\]
otherwise
\[S=B_{S_1,S_2}.\]

If $S$ has no zero-element, consider $S^0=S\cup\{0\}$. If it were of type B the above proof shows that necessarily $(v,v)\in S^0$ or $(v,u),(u,v)\in S^0$ for some $v,u\in S$ but then $0\in S$, contradiction. Thus it is of type A and necessarily $F\cap G=\emptyset$.

Finally $S$ is an intersection of constructible semigroups, since for any sets $F,G,F^\prime,G^\prime$,
\[B_{F,G}\cap B_{F^\prime,G^\prime}=B_{F\cap F^\prime,G\cap G^\prime}\] and if $F$ and $G$ are infinite we may always write them as an intersection of cofinite sets. If $S$ is $\infty$-constructible then $F$ and $G$ must be finite or co-countable by Lemma \ref{L:F,G as const}, as needed.
\end{proof}

\subsection{Subgroups of $PGL_2(\mathbb{C})$}
Since every $\infty$-constructible subgroup of $PGL_2(\mathbb{C})$ is closed (Fact \ref{F:subconstgrp-of-alg}), we may use the following characterization:

\begin{fact} \label{F:subgroup of pgl2}\cite{Nguyen}
Let $\gamma:SL_2(\mathbb{C})\to PSL_2(\mathbb{C})$ be the canonical projection and 
$$B=\left\{  
\begin{pmatrix}
a & b \\
0 & a^{-1}
\end{pmatrix}
: a\in\mathbb{C}^*,b\in\mathbb{C} \right\} $$
and
$$D_\infty=\left\{ 
\begin{pmatrix}
c & 0\\
0 & c^{-1}
\end{pmatrix}
: c\in \mathbb{C}^* \right\}
\cup
\left\{ 
\begin{pmatrix}
0 & -d\\
d^{-1} & 0
\end{pmatrix}
: d\in \mathbb{C}^* \right\} $$
be the Borel subgroup and the infinite dihedral subgroup of $SL_2(\mathbb{C})$.
Every algebraic subgroup of $PGL_2(\mathbb{C})$ is, up to conjugation, one of the following:
\begin{enumerate}[(1)]
\item $PGL_2(\mathbb{C})$;
\item a subgroup of $\gamma(B)$;
\item $\gamma(D_\infty)$;
\item $D_n$ (the dihedral group of order $2n$), $A_4$ (the tetrahedral group), $S_4$ (the octahedral group), or $A_5$ (the icosahdral group).
\end{enumerate}
\end{fact}

\subsection{Submonoids of $PM_2(\mathbb{C})$}
Let $M\subseteq PM_2(\mathbb{C})$ be a submonoid and $M=H\cup S$ be a decomposition of $M$ to singular and regular parts, respectively. Recall that we wrote rank $1$ elements of $PM_2(\mathbb{C})$ as $(v,u)$ where $v$ corresponds to the image and $u$ to the kernel. Hence, if $g\in PGL_2(\mathbb{C})$ and $(v,u)\in PM_2(\mathbb{C})$ is of rank $1$,
\[g(v,u)=(g^{-1}v,u)\text{ and } (v,u)g=(v,gu).\]
This gives the following

\begin{lemma}
Let $H$ by a subgroup of $PGL_2(\mathbb{C})$. \[H\cup B_{F,G}\] is a submonoid if and only if $F$ and $G$ are $H$-invariant. Similarly, 
\[H\cup B_{F,\{v\}}\cup B_{\{v\},G}\] is a submonoid if and only if $F$, $G$ and $\{v\}$ are $H$-invariant.
\end{lemma}

Either way, it is important to understand orbits of actions by algebraic subgroups of $PGL_2(\mathbb{C})$ on $\mathbb{C}P^1$. The following easy to check result describes these orbits and is a direct computation using Fact \ref{F:subgroup of pgl2}. We note that it may also be reached using model theoretic tools (see \cite{almost-orth}).

\begin{lemma}\label{L:finite or cofinite orbits}
Every algebraic subgroup of $PGL_2(\mathbb{C})$ has finite or cofinite orbits (acting on $\mathbb{C}P^1$). Furthermore, if the subgroup is not finite, the number of finite orbits is finite and there is one infinite orbit.
\end{lemma}

\begin{proposition}\label{P:submonoid of PM2}
Every $M=H\cup S$ $\infty$-constructible submonoid of $PM_2(\mathbb{C})$ is a countable intersection of constructible monoids. In fact, if the regular part is infinite then the monoid is constructible.
\end{proposition}
\begin{remark}
Recall that the regular part of an $\infty$-constructible submonoids of $PM_2(\mathbb{C})$ is an algebraic group and in particular constructible.
\end{remark}
\begin{proof}
$S$ must be either of type A or of type B, let $F$ and $G$ be as in the definitions (see Proposition \ref{P:first}).

If $H$ is infinite then, since $F$ and $G$ are $H$-invariant, by Lemma \ref{L:finite or cofinite orbits}, they must be finite or cofinite. Thus $M$ is a finite union of constructible sets and hence constructible.

If $H$ is finite then, as in Proposition \ref{P:first}, if $F$ (resp. $G$) is not finite it must be co-countable. Thus $F^c$  (resp. $G^c$) is a countable union of finite orbits. Assume \[F=\bigcap_i V_i \text{ and } G=\bigcap_i U_i,\] where $V_i$ and $U_i$ are co-finite (or finite) and $H$-invariant, thus \[M=\bigcap_i (H\cup B_{U_i,V_i})\] if $S$ is of type A and 
\[M=\bigcap_i (H\cup B_{V_i,\{v\}}\cup B_{\{v\},U_i})\] if $S$ is of type B.
\end{proof}

\section{Submonoids of $M_2(\mathbb{C})$}\label{s:M_2}
Every $\infty$-constructible submonoid of $M_2(\mathbb{C})$ may be decomposed as $H\cup S$ where $H$ is an algebraic subgroup (using Fact \ref{F:subconstgrp-of-alg}) of $GL_2(\mathbb{C})$ and $S\subseteq M_2^0(\mathbb{C})$. As before, we start by understanding the latter.

\subsection{Multiplicities and Subsemigroups of $M^0_2(\mathbb{C})$}
Let $S\subseteq M^0_2(\mathbb{C})$ be an $\infty$-constructible submonoid.

Consider the map $\pi :M_2(\mathbb{C})\to PM_2(\mathbb{C})$, for every $x\in PM_2(\mathbb{C})$, $\pi|_S^{-1}(x)$ is a set of the form $\{zx:z\in Z_x\subseteq \mathbb{C}^\times\}$ for some $Z_x$. In general we might have different $Z_x$ for different $x$. We will need to understand how the $Z_x$ behave, when the $x$ varies.
\begin{definition}
For each $x\in S$ we define the \emph{multiplicity of $x$} to be the set \[Z_x=\{z\in \mathbb{C}^\times: zx\in S\}.\]
\end{definition}

It is an $\infty$-constructible subset of $\mathbb{C}^\times$. 
\begin{remark}
Notice that for $x,y\in S$, $\pi(x)=\pi(y)$ if and only if there exists $\lambda\in \mathbb{C}^\times$ such that $x=\lambda y$.
\end{remark}

Recall that $0$ is not considered an idempotent. Some basic properties:

\begin{lemma}\label{L:basic-prop-mult}
\begin{enumerate}
\item If $x\in S$ is an idempotent then $Z_x$ is an algebraic subgroup of $\mathbb{C}^\times$, i.e. a subgroup generated by a primitive root of unity or all of $\mathbb{C}^\times$.
\item If $e\in S$ is an idempotent then $\pi(e)$ is an idempotent. Conversely, if $(v,u)\in \pi (S)$ with $v\neq u$ then there exists an idempotent $e\in S$ with $\pi(e)=(v,u)$.
\item If $lxr=y$, for $x,y,l,r\in S$ then $Z_x\subseteq Z_y$.
\item For every idempotent $e\in S$ and $\lambda\in Z_e$, $Z_e=Z_{\lambda e}$.
\item $Z_0=\mathbb{C}^\times$.
\end{enumerate}
\end{lemma}
\begin{proof}
(1). If $x$ is an idempotent then $Z_x$ is an $\infty$-constructible subsemigroup of $\mathbb{C}^\times$, hence, by Fact \ref{F:subsconstsemi-of-alg}, an algebraic subgroup of $\mathbb{C}^\times$.

(2). Let $(v,u)\in \pi(S)$ with $v\neq u$ and let $e\in M_2^0(\mathbb{C})$ be an idempotent such that $\pi(e)=(v,u)$. There exists $\lambda\in \mathbb{C}^\times$ such that $\lambda e\in S$. Consider \[\{z\in \mathbb{C}^\times: ze\in S\},\] it is an $\infty$-constructible subsemigroup of $\mathbb{C}^\times$ hence, as before, an algebraic group so $e\in S$.

The rest is clear.
\end{proof}

Recall that since $S$ is an $\infty$-constructible subsemigroup of $M_2^0(\mathbb{C})$, by Lemma \ref{L:gen-of-chev} $\pi(S)$ is an $\infty$-constructible subsemigroup of $PM_2^0(\mathbb{C})$ so we may use Proposition \ref{P:first}.

\begin{proposition}\label{P:multiplicity}
If $\pi(S)$ is of type A, i.e. of the form $B_{F,G}$, then
\begin{enumerate}
\item if $e,f\in E(S)$ then $Z_e=Z_f$;
\item if $|F|>1$ and $|G|>1$ then for every $0\neq n\in S$ nilpotent and $e\in E(S)$ idempotent, $Z_n=Z_e$;
\item if $|F|=1$ or $|G|=1$ then up to multiplication by an element of $\mathbb{C}^\times$ there is at most one nilpotent and for any nilpotent $n\in S$ and idempotent $e\in S$ we have $Z_e\subseteq Z_n$. Furthermore the nilpotents form an ideal of $S$.
\end{enumerate}
\end{proposition}
\begin{proof}
\begin{enumerate}
\item Assume that $\pi(e)=(v,u),\, \pi(f)=(v^\prime,u^\prime)$, where $v\neq u,v^\prime\neq u^\prime$. Since
\[(v^\prime,u)(v,u)(v,u^\prime)=(v^\prime,u^\prime),\]
\[(v,u^\prime)(v^\prime,u^\prime)(v^\prime,u)=(v,u)\]
and $(v,u^\prime),(v^\prime,u)\in B_{F,G}$, there exist $x,y\in S$ and $z\in Z_f,\, l\in Z_e$ such that $xey=zf$ and $yfx=le$. The result follows by Lemma \ref{L:basic-prop-mult}(3). 
\item If $\pi(n)=(v,v)$ then for any $v^\prime\neq v$ and $v\neq u$ since
\[(v,u)(v,v)=(v,v)\]
and
\[(v,v)(v^\prime,u)=(v,u)\]
there exist $z\in Z_n$, $x\in S$ and $l\in Z_e$ such that 
\[en=zn\]
and
\[nx=le.\]
Thus $zn=en=e\cdot en=zen=z^2n$ and since $n\neq 0$, $z=1$. So $Z_e\subseteq Z_n$. The other direction follows since $Z_{le}=Z_e$.
\item Assume that $F=\{v_0\}$. Thus, for every nilpotent $n\in S$, $\pi(n)=(v_0,v_0)$ and 
\[(v_0,u)(v_0,v_0)=(v_0,v_0)\]
for $v_0\neq u$, thus as was done in $(2)$, $en=n$ and $Z_e\subseteq Z_n$.
Furthermore,  $ne=0$ for every $e\in E(S)$. The result follows since every element of $M_2^0(\mathbb{C})$ is either a nilpotent or a multiple of an idempotent by an element of $\mathbb{C}$.
\end{enumerate}
\end{proof}

\begin{example}
The requirement that $\pi(S)$ be of the form $B_{F,G}$ is necessary. For example 
\[\mathbb{C}^\times\cdot \begin{pmatrix}
1 & 0\\
0 & 0
\end{pmatrix}
\cup \mu_6\cdot \begin{pmatrix}
0 & 0\\
0 & 1
\end{pmatrix}
\cup \{0\},\]
where $\mu_6$ is the subgroup of $\mathbb{C}^\times$ of units of order $6$, is a constructible monoid, but $Z_{\left( \begin{smallmatrix}
1 & 0\\
0 & 0
\end{smallmatrix}\right) }\neq Z_{\left( \begin{smallmatrix}
0 & 0\\
0 & 1
\end{smallmatrix}\right) }$.
\end{example}

The above proposition gives a lot of information about how the multiplicity varies. With the aid of some calculations we can say more.

It is an easy exercise to see that every non-zero idempotent of $M_2^0(\mathbb{C})$ has the form \[\frac{1}{ad-bc}\begin{pmatrix}
ad & -ac\\
bd & -bc
\end{pmatrix}\]
and every non-zero nilpotent has the form
\[\lambda\begin{pmatrix}
ab & -a^2\\
b^2 & -ab\end{pmatrix}.\] Furthermore, their images in $PM_2^0(\mathbb{C})$ are $([a:b],[c:d])$ and $([a:b],[a:b])$, respectively.

A product of two idempotents in $M_2^0(\mathbb{C})$ is either a multiple of an idempotent by an element of $\mathbb{C}$ or a nilpotent.  If it is a multiple of an idempotent we would like to calculate this number.

\begin{lemma}\label{L:lambda}
\begin{enumerate}
\item Let $e,f,h\in E(S)$ and $\lambda\in\mathbb{C}$ with $ef=\lambda h$. If
\[\pi(e)=\left([a:b],[c:d]\right),
\pi(f)=\left([x:y],\, [z:w]\right)\text{and } \pi(h)=\left([a:b],[z:w]\right) \]
then 
\[\lambda=\frac{(aw-bz)(dx-cy)}{(ad-bc)(xw-yz)}.\]

\item Given $[a:b],[x:y],[z:w]\in \mathbb{C}P^1$, pairwise distinct, the map
\[[c:d]\mapsto \lambda=\frac{(aw-bz)(dx-cy)}{(ad-bc)(xw-yz)}\]
from $\mathbb{C}P^1\setminus \{[a:b],[x:y]\}$ to $\mathbb{C}^\times$
is bijective.
\end{enumerate}
\end{lemma}
\begin{proof}
\begin{enumerate}
\item Since
\[e=\frac{1}{ad-bc}\begin{pmatrix}
ad & -ac\\
bd & -bc
\end{pmatrix}\text{ and }
f=\frac{1}{xw-yz}\begin{pmatrix}
xw & -xz\\\
yw & -yz
\end{pmatrix}, \]
multiplication yields
\[ef=\frac{(aw-bz)(dx-cy)}{(ad-bc)(xw-yz)}\cdot
\frac{1}{aw-bz}\begin{pmatrix}
aw & -az\\
bw & -bz
\end{pmatrix}=\lambda h.\]
Indeed, $\lambda$ is independent of our choice of representatives for $$[a:b],[c:d],[x:y] \text{ and } [z:w].$$

\item  Injectivity: since $[a:b]\neq [x:y]$, $[c:d]=[c^\prime :d^\prime]$,
\[\frac{(aw-bz)(dx-cy)}{(ad-bc)(xw-yz)}=\frac{(aw-bz)(d^\prime x-c^\prime y)}{(ad^\prime-bc^\prime)(xw-yz)}\]
implies that
\[(ay-bx)(dc^\prime-c^\prime d)=0.\]

Surjectivity: given $\lambda\in\mathbb{C}^\times$, solving
$$\lambda=\frac{(aw-bz)(dx-cy)}{(ad-bc)(xw-yz)},$$
for $c,d$, is equivalent to solving the following homogeneous linear equation:
$$\left[\lambda b(xw-yz)-y(aw-bz)\right] c+\left[x(aw-bz)-a\lambda (xw-yz)\right] d=0.$$
This equation always has a (projective) solution.
\end{enumerate}
\end{proof}

\begin{definition}
Let $A\subseteq M_2^0(\mathbb{C})$ be a subset. If $Z_x=Z_y$ for every $x,y\in A\setminus \{0\}$ then we will say that $A$ has \emph{equal multiplicity} and we will denote its multiplicity by $Z_A$.
\end{definition}

\begin{proposition}\label{P:F^c>1 or =1}
Let $S$ be an $\infty$-constructible monoid with $\pi(S)=B_{F,G}$ (of type A). 
\begin{enumerate}
\item If $|F|>1$ and $|G|$ is infinite (or $|G|>1$ and $|F|$ is infinite) then $S$ has equal multiplicity $Z_S=\mathbb{C}^\times$;

\item If $|F|=1$ (or $|G|=1$) then a product of idempotents is an idempotent.
\end{enumerate}
\end{proposition}
\begin{proof}

\begin{enumerate}
\item Let $F=\{[a:b],[x:y],\ldots\}$ ($[a:b]\neq [x:y]$) and let \[[z:w]\in G\setminus \{[a:b],[x:y]\}.\] By Lemma \ref{L:basic-prop-mult} there exists $h\in E(S)$ with $\pi(h)=\left([a:b],[z:w]\right)$. 

The injectivity from Lemma \ref{L:lambda} implies that
\[|G\setminus \{ [a:b],[x:y] \}| \leq |Z_h|.\]
Since $G$ is infinite and $Z_h$ is an algebraic subgroup of $\mathbb{C}^\times$ (Lemma \ref{L:basic-prop-mult}(1)), necessarily $Z_h=\mathbb{C}^\times$. By Proposition \ref{P:multiplicity}, $Z_S=\mathbb{C}^\times.$

\item Let $e,f\in E(S)$ with 
\[\pi(e)=\left([a:b],[c:d]\right) \text { and }
\pi(f)=\left([a:b],[z:w]\right).\]
By Lemma \ref{L:lambda}, $ef=h$ where $h\in E(S)$ and $\pi(h)=\left([a:b],[z:w]\right)$.

\end{enumerate}
\end{proof}


There is an inherent problem with the nilpotents of a semigroup $S\subseteq M_2^0(\mathbb{C})$. For instance if $\pi(S)=B_{F,G}$ and we know the multiplicity of an idempotent we know all the multiples of idempotents lying in $S$. If $(v,u)\in B_{F,G}$ with $v\neq u$ then there exits $e\in S$ with $\pi(s)=(v,u)$ and it is uniquely defined by knowing $u$ and $v$. Since the multiplicities of all the idempotents are equal this gives us a complete description. The picture is different for nilpotents. For instance, for any $\lambda\in \mathbb{C}^\times$ the following is a semigroup \[\left \{ \mu\begin{pmatrix}
0 & \lambda\\
0 & 0
\end{pmatrix}:\mu^5=1\right \} \cup \{0\}\] and the multiplicity of $Z_{\left(\begin{smallmatrix}
0 & \lambda\\
0 & 0
\end{smallmatrix}\right)}$ does not depend on $\lambda$, furthermore they all have the exact same image under $\pi$, and thus not determined by $\lambda$. 

As a result of the above discussion we may set the following notations,
\begin{notation}
If $\pi(S)=B_{F,G}$ has no nilpotents we shall write $S=Z_SB_{F,G}$ and if $B_{F,G}$ has equal multiplicity $\mathbb{C}^\times$ we shall write $S=\mathbb{C}^\times B_{F,G}:=\pi^{-1}(B_{F,G}).$
\end{notation}

\begin{proposition}\label{P:M_2^0-type A}
Let $S$ be an $\infty$-constructible subsemigroup of $M^0_2(\mathbb{C})$ with $\pi(S)=B_{F,G}$ (of type A).
\begin{enumerate}
\item If $S$ has no nilpotents then it has equal multiplicity and \[S=Z_SB_{F,G}=\bigcup_{e\in S\text{ idempotent}} Z_S\cdot e.\] 
If $B_{F,G}$ is infinite  and $|F|,|G|>1$ then $Z_S=\mathbb{C}^\times$.
\item If $S$ has nilpotents then 
\begin{enumerate}
\item If $|F|>1$ and $|G|>1$ then $S$ has equal multiplicity, so \[S=\bigcup_{e\in S\text{ idempotent}} Z_S\cdot e\cup \bigcup_{n\in S\text{ nilpotent}} n.\] If $B_{F,G}$ is infinite then it has equal multiplicity $\mathbb{C}^\times$, so \[S=\pi^{-1}(B_{F,G})=\mathbb{C}^\times B_{F,G}.\]

\item If $|F|=1$ (or $|G|=1$) then $S$ if of the form
\[S=Z_eB_{\{v\},G\setminus\{v\}}\cup Z_n\cdot n,\] where $e$ is any idempotent and $n$ is any nilpotent.

\end{enumerate}
\end{enumerate}
If $S$ has no zero element, then only $(1)$ applies.

Furthermore, $S$ is an intersection of constructible semigroups.
\end{proposition}
\begin{proof}

\begin{enumerate}
\item By Proposition \ref{P:multiplicity}, $S$ has equal multiplicity $Z_S$.  If $B_{F,G}$ is infinite and $|F|,|G|>1$, then by Proposition \ref{P:F^c>1 or =1} $Z_S=\mathbb{C}^\times$.

If $B_{F,G}$ is finite then either $S$ is finite (and hence constructible) or $S=\pi^{-1}(B_{F,G})$. If $B_{F,G}$ is infinite then by Proposition \ref{P:first}, $B_{F,G}=\bigcap_i B_{V_i,U_i}$ with $B_{V_i,U_i}$ constructible. If $|F|,|G|>1$ then \[S=\bigcap_i \pi^{-1}(B_{V_i,U_i}).\] 
Note that $\pi^{-1}(B_{V_i,U_i})$ is constructible by Lemma \ref{L:gen-of-chev}.
On the other hand, if $F=\{v\}$ write $B_{\{v\},G}=\bigcap_i B_{\{v\},U_i}$, with $B_{\{v\},U_i}$ constructible. By Proposition \ref{P:F^c>1 or =1} a product of idempotents in $Z_SB_{\{v\},U_i}$ is still an idempotent and hence a semigroup so \[S=\bigcap_iZ_SB_{\{v\},U_i}.\] 
Note that since $Z_S$ is a constructible subset of $\mathbb{C}$, \[Z_SB_{\{v\},U_i}=Z_S\cdot \left( \pi^{-1}(B_{\{v\},U_i})\cap \{x:x^2=x\}\right)\] is also constructible.
\item \begin{enumerate}
\item By Proposition \ref{P:multiplicity}, $S$ has equal multiplicity and if $B_{F,G}$ is infinite then $Z_S=\mathbb{C}^\times$ by Proposition \ref{P:F^c>1 or =1}.

The proof is as is in $(1)$, even simpler because $|F|,|G|>1$.

\item By Proposition \ref{P:multiplicity} there is only one nilpotent (up to multiplicity).

Denote by $N:=Z_n\cdot n$ the ideal of nilpotents (Proposition \ref{P:multiplicity}). Since a product of idempotents is an idempotent (again Proposition \ref{P:multiplicity}), $S\setminus N$ is either empty or a subsemigroup. By $(1)$ there exist constructible semigroups such that \[S\setminus N=\bigcap_i M_i.\]
Since $Z_n$ is $\infty$-constructible, $N=\bigcap_i A_i\cdot n$ where the $A_i$ are constructible subsets of $\mathbb{C}^\times$. Since $M_j\cap (A_i\cdot n)=\emptyset$, \[S=\bigcap_i (M_i\cup A_i\cdot n).\] To show that the intersectants are semigroups, observe that the $\pi(M_i)$ are of the form $B_{\{v\},U_i}$ and hence either by Proposition \ref{P:multiplicity} or by direct calculation, for every idempotent $e\in M_i$ and $an\in A_i\cdot n$, \[ean=an\text{ and } ane=0,\] so it is indeed an intersection of constructible semigroups.
\end{enumerate}
\end{enumerate}
\end{proof}

Before we handle semigroups $S$ with $\pi(S)$ of type B, observe the following easy lemma:

\begin{lemma}\label{L:cal}
Let $S$ be a subsemigroup of $M_2^0(\mathbb{C})$ and $e,f\in E(S)$. If $\pi(e)=([a:b],[c:d])$ and $\pi(f)=([x:y],[a:b])$ with $[c:d]\neq [x:y]$ then 
\[ef=\frac{(cy-dx)}{(ad-bc)(ay-bx)}\begin{pmatrix}
ab & -a^2\\
b^2 & -ad
\end{pmatrix}.\]
\end{lemma}

\begin{proposition}\label{P:M_2^0-type B}
Let $S$ be an $\infty$-constructible subsemigroup of $M^0_2(\mathbb{C})$ with $\pi(S)=B_{F,\{v\}}\cup B_{\{v\},G}$ (of type B).
Then \[S=Z_eB_{\{v\},G\setminus \{v\}}\cup Z_f B_{F\setminus \{v\},\{v\}}\cup Z_n\cdot n,\] where $e$ and $f$ are idempotents with $\pi(e)\in B_{\{v\},G}$ and $\pi(f)\in B_{F,\{v\}}$ and $n$ is any nilpotent.

Furthermore, $S$ is an intersection of constructible semigroups.
\end{proposition}
\begin{proof}
There is only one nilpotent up to multiplicity so the structure follows from similar arguments as in Proposition \ref{P:M_2^0-type A}.

By Proposition \ref{P:M_2^0-type A}, there exist definable semigroups $L_i$ and $R_i$ such that \[Z_eB_{\{v\},G\setminus \{v\}}=\bigcap_i L_i\text{ and}\]  \[Z_f B_{F\setminus \{v\},\{v\}}=\bigcap_i R_i.\] 
We may obviously choose the $L_i$ and $R_i$ to be such that for every $x\in L_i$ and $y\in R_i$, $xy$ is nilpotent and up to multiplicity the same nilpotent as in $S$. 

\begin{claim}
The set of multiplicities of the nilpotent we get when multiplying an element from $L_i$ with an element from $R_i$ is constructible.
\end{claim}
\begin{claimproof}
One can either use Lemma \ref{L:cal}, or the fact that definable sets correspond to constructible sets (Fact \ref{F:QE}). Another approach, which is similar to the latter, is to look at the subset $\{(x,y,z):xy=zn\}$ of $L_i\times R_i\times \mathbb{C}$, where $n$ is any one of the nilpotents of $S$. It is obviously a constructible subset. Taking the projection on the last coordinate and using Lemma \ref{L:gen-of-chev}, we get our result.
\end{claimproof}

We may thus choose $Z_n\cdot n=\bigcap_i A_i\cdot n$ with $A_i$ constructible and containing the different multiplicities we get from these products. 

Since $R_i, L_i$ and $A_i\cdot n$ are pairwise disjoint (we may choose the $R_i$ and $L_j$ not to have nilpotents and hence they are disjoint)
\[S=\bigcap_i(L_i\cup R_i\cup A_i\cdot n).\]
Each intersectant is a semigroup since if $e\in L_i$ and $f\in R_i$ then direct calculation (or Proposition \ref{P:multiplicity}) shows that
\[fe=0,\, ne=0,\, fn=0,\, en=n,\, nf=n\]
and $ef\in A_i\cdot n$ by the choice of the $A_i$.
\end{proof}

\subsection{Submonoids of $M_2(\mathbb{C})$}
First, this easy lemma:
\begin{lemma}\label{L:easy lemma}
Let $H\subseteq GL_2(\mathbb{C})$ be a subgroup and $S\subseteq M^0_2(\mathbb{C})$ a subsemigroup.
If $\pi(H\cup S)$ is a monoid and $S$ has equal multiplicity $\mathbb{C}^\times$ then $H\cup S$ is a monoid.  
\end{lemma}
\begin{proof}
Let $h\in H$ and $s\in S$. Since $\pi(H\cup S)$ is a monoid, there exists $\lambda\in\mathbb{C}^\times$ such that $\lambda hs\in S$. By assumption, $\lambda^{-1}\in Z_{\lambda hs}$ so $hs\in S$.
\end{proof}

Let $M=H\cup S$ be an $\infty$-constructibl submonoid of $M_2(\mathbb{C})$, where $H$ and $S$ are the regular and singular parts, respectively. By the previous sections, $H$ is an algebraic group and $S$ is an intersection of definable semigroups. Furthermore, we have $\pi(M)=\pi(H)\cup \pi(S)$, where $\pi(H)$ is constructible and hence an algebraic subgroup of $PGL_2(\mathbb{C})$.

\begin{remark}
Algebraic subgroups of $GL_2(\mathbb{C})$ were treated in \cite{Nguyen}.
\end{remark}

\begin{lemma}\label{L:H-acting-on-nilpotent}
Let $H$ be an infinite algebraic subgroup of $GL_2(\mathbb{C})$ and $n\in M_2(\mathbb{C})$ such that for every $h\in H$ there exists $\lambda_h\in \mathbb{C}^\times$ such that $hn=\lambda_hn$. Then \[\{\lambda_h :h\in H\}=\mathbb{C}^\times.\]
\end{lemma}
\begin{proof}
Since $G=\{\lambda_h:h\in H\}$ is an infinite constructible subgroup of $\mathbb{C}^\times$, by Fact \ref{F:subconstgrp-of-alg} it must be all of $\mathbb{C}^\times$.
\end{proof}

\begin{proposition}\label{P:M_2-type A}
Let $M=H\cup S$ be an $\infty$-constructible submonoid of $M_2(\mathbb{C})$ with $\pi(S)=B_{F,G}$ (of type A). $M$ is an intersection of constructible monoids. Moreover, in the following cases $M$ is necessarily constructible:
\begin{list}{•}{}
\item $\pi(H)$ is infinite.
\item $B_{F,G}$ does not have exactly one nilpotent and $S$ does not have equal multiplicity $\mathbb{C}^\times$.
\end{list}

\end{proposition}
\begin{proof}
We use the characterization of $S$ given in Proposition \ref{P:M_2^0-type A} and break the proof into distinct cases:
\begin{enumerate}
\item Assume $S$ has equal multiplicity $\mathbb{C}^\times$, thus $M=H\cup \mathbb{C}^\times B_{F,G}$. If $B_{F,G}$ is constructible, for instance if $\pi(H)$ is infinite (Proposition \ref{P:submonoid of PM2}), then $M$ is constructible. Otherwise, as was done in Proposition \ref{P:submonoid of PM2} we may write \[\pi(M)=\bigcap_i (\pi(H)\cup B_{U_i,V_i}),\] where the $U_i, V_i$ are co-finite and $\pi(H)$-invariant. Thus \[M=\bigcap_i (H\cup \mathbb{C}^\times B_{U_i,V_i})\] and the intersectants are monoids by Lemma \ref{L:easy lemma}.

\item Assume $S$ does not have equal multiplicity $\mathbb{C}^\times$.
\begin{enumerate}
\item Assume $B_{F,G}$ does not have exactly one non-zero nilpotent (i.e. $S$ does not have exactly one non-zero nilpotent, upto multiplicity), thus necessarily $B_{F,G}$ is finite, $S$ has equal multiplicity and $E(S)\neq \emptyset$. Since the multiplicity of any idempotent is an algebraic group, $M=H\cup S$ is constructible.
\item Otherwise, $S$ is of the form \[S=Z_eB_{\{v\},G\setminus \{v\}}\cup Z_n\cdot n,\] where $e$ is an idempotent and $n$ is a nilpotent. Since $hn\in Z_n\cdot n$ for $h\in H$, $H$ acts on $Z_n$ by sending $z$ to $z_h$ where $z_h\cdot n=h\cdot (zn)$. If $B_{\{v\},G\setminus\{v\}}$ is constructible and $H$ is infinite (for instance if $\pi(H)$ is infinite), then by Lemma \ref{L:H-acting-on-nilpotent}, $Z_n=\mathbb{C}^\times$, so $M=H\cup S$ is constructible.

Either way, we may write $Z_n=\bigcap_i A_i$, where the $A_i$ are constructible and $H$-invariant. Indeed, since $Z_n$ is an $\infty$-constructible subset of $\mathbb{C}$, by Lemma \ref{L:cons-in-curve} it is either finite or co-countable. If $\mathbb{C}^\times \setminus Z_n$ is countable it is a union of countable many $H$-orbits and each of them is constructible. Since \[\pi(H)\cup B_{\{v\},G\setminus \{v\}}=\bigcap_i (\pi(H)\cup B_{\{v\},U_i}),\] where $v\notin U_i$ are co-finite and $\pi(H)$-invariant, \[M=\bigcap_i (H\cup Z_eB_{\{v\},U_i}\cup A_i\cdot n).\]
Using the final arguments in the proof of Proposition \ref{P:M_2^0-type A} and the choice of the $A_i$, in order to verify that the intersectants are monoids we only need to verify that $Z_eB_{\{v\},U_i}$ is $H$-invariant. This follows by a similar argument to the one that was given in Lemma \ref{L:easy lemma}.
\end{enumerate}
\end{enumerate}
\end{proof}

The following is an example of an $\infty$-constructible semigroup which is not constructible.
\begin{example}
Let $Z$ be an co-countable subset of $\mathbb{C}$ and $n$ any non-zero nilpotent of $M_2(\mathbb{C})$. Since $Z=\bigcap_{a\notin Z} \mathbb{C}\setminus \{a\}$, $Z$ is $\infty$-constructible, but not constructible. As a result, the semigroup $S=Z\cdot n\cup\{0\}$ is not constructible.
\end{example}

\begin{proposition}\label{P:M_2-type B}
Let $M=H\cup S$ be an $\infty$-constructible submonoid of $M_2(\mathbb{C})$ with $\pi(S)=B_{F,\{v\}}\cup B_{\{v\},G}$ (of type B). $M$ is an intersection of constructible monoids. Moreover, if $\pi(H)$ is infinite then $M$ is constructible.
\end{proposition}
\begin{proof}
Following Proposition \ref{P:M_2^0-type B}, \[S=Z_eB_{\{v\},G\setminus \{v\}}\cup Z_fB_{F\setminus\{v\},\{v\}}\cup Z_n\cdot n.\] Using Proposition \ref{P:M_2-type A}, we may write 
\[Z_eB_{\{v\},G\setminus \{v\}}=\bigcap_i L_i \text{ and}\] 
\[Z_fB_{F\setminus\{v\},\{v\}}=\bigcap_i R_i\] where the $L_i$ and $R_i$ are $H$-invariant. As in Proposition \ref{P:M_2-type A}, $H$ acts on $Z_n$. Write $Z_n=\bigcap_i A_i$ where the $A_i$ are $H$-invariant and contain the different multiplicities of $n$ we get when multiplying $xy$ for $x\in L_i$ and $y\in R_i$ (see the proof of Proposition \ref{P:M_2^0-type B}).
Thus \[M=\bigcap_i (H\cup L_i\cup R_i \cup A_i\cdot n).\] Similarly to the argument in the proof Proposition \ref{P:M_2^0-type B}, and by choice, the intersectants are monoids.

As in Proposition \ref{P:M_2-type A}, if $\pi(H)$ is infinite, $M=H\cup S$ is constructible.
\end{proof}

\paragraph*{Acknowledgements}
I would like to thank my PhD adviser, Ehud Hrushovski for his ideas and discussions leading to and during this paper.

\bibliographystyle{plain}
\bibliography{M2(C)}

\end{document}